\documentclass[12pt]{article}
\usepackage{exscale,relsize}
\usepackage{amsmath}
\usepackage{amsfonts}
\usepackage[hidelinks]{hyperref}
\usepackage{amssymb}
\usepackage{calc}
\usepackage{theorem}
\usepackage{pifont}      
\usepackage{array}
\usepackage{color}
\usepackage{enumerate}
\usepackage{bbm}
\usepackage{graphicx}
\usepackage{float}
\usepackage{subfigure}

\newcommand{\scal}[2]{\langle{{#1},{#2}}\rangle}

\newcommand{\HH}{\ensuremath{\mathcal H}}

\newcommand{\RR}{\ensuremath{\mathbb R}}
\newcommand{\SSS}{\ensuremath{\mathbb S}}

\newcommand{\NN}{\ensuremath{\mathbb N}}

\newcommand{\menge}[2]{\big\{{#1} \mid {#2}\big\}}

\newcommand{\dom}{\ensuremath{\operatorname{dom}}}

\newcommand{\gra}{\ensuremath{\operatorname{gra}}}

\newcommand{\inte}{\ensuremath{\operatorname{int}}}

\newcommand{\ran}{\ensuremath{\operatorname{ran}}}

\newcommand{\Fix}{\ensuremath{\operatorname{Fix}}}

\newcommand{\Id}{\ensuremath{\operatorname{Id}}}

\renewcommand{\phi}{\ensuremath{\varphi}}

\newtheorem{theorem}{Theorem}[section]
\newtheorem{lemma}[theorem]{Lemma}
\newtheorem{fact}[theorem]{Fact}
\newtheorem{corollary}[theorem]{Corollary}
\newtheorem{proposition}[theorem]{Proposition}
\newtheorem{definition}[theorem]{Definition}

\theoremstyle{plain}{\theorembodyfont{\rmfamily}
}
\theoremstyle{plain}{\theorembodyfont{\rmfamily}
}
\theoremstyle{plain}{\theorembodyfont{\rmfamily}
}
\theoremstyle{plain}{\theorembodyfont{\rmfamily}
\newtheorem{example}[theorem]{Example}}
\theoremstyle{plain}{\theorembodyfont{\rmfamily}
\newtheorem{remark}[theorem]{Remark}}
\theoremstyle{plain}{\theorembodyfont{\rmfamily}
}

\newcommand{\pluss}{{\hskip1pt \raise1pt\vbox{\hrule width6pt \vskip1pt
\hrule width6pt}\kern-4pt{\lower1pt\hbox{\vrule height6pt \kern1pt\vrule
height6pt}}\hskip5pt}}

\usepackage{tikz}

\begin{document}

\title{{\fontfamily{ptm}\selectfont On Bauschke-Bendit-Moursi modulus of averagedness and
 classifications of averaged nonexpansive operators}}

\author{
Shuang\ Song\thanks{Department of Mathematics, I.K. Barber Faculty of Science,
The University of British Columbia, Kelowna, BC Canada V1V 1V7. E-mail:  \texttt{cat688@student.ubc.ca}.}~ and ~ Xianfu Wang\thanks{Department of Mathematics, I.K. Barber Faculty of Science,
The University of British Columbia, Kelowna, BC Canada V1V 1V7. E-mail:  \texttt{shawn.wang@ubc.ca}.} }

\date{June 24, 2025}

\maketitle

\vskip 8mm

\begin{abstract} \noindent Averaged operators are important in Convex Analysis and Optimization Algorithms.
In this paper, we propose classifications of averaged operators, firmly nonexpansive operators, and proximal operators using the Bauschke-Bendit-Moursi modulus of averagedness. We show that if an operator is averaged with
a constant less than $1/2$, then it is a bi-Lipschitz
homeomorphism.
Amazingly the proximal operator of a convex function has its modulus of averagedness less than $1/2$ if and only if the function is Lipschitz smooth.
Some results on the averagedness of operator compositions are obtained.
Explicit formulae for calculating the modulus of averagedness of resolvents and proximal operators
in terms of various values associated with the maximally monotone operator or subdifferential are also
given. Examples are provided to illustrate our results.
\end{abstract}

{\small



\section{Introduction}
Throughout we assume that
$$
X \text { is a real Hilbert space with inner product }\langle\cdot, \cdot\rangle: X \times X \rightarrow \mathbb{R} \text {, }
$$ and induced norm $\|\cdot\|$. Let $\Id$ denote the identity operator on $X$.
Recall the following well-known definitions \cite{BC, cegielski}.
\begin{definition}
Let $T: X \rightarrow X$ and $\mu>0$. Then $T$ is
\begin{enumerate}
\item nonexpansive\footnote{For convenience, we shall assume that
$T$ has a full domain throughout the paper while
one can generalize it to be on a proper subset of $X$.} if $$
(\forall x \in X)(\forall y \in X) \quad\|T x-T y\| \leqslant\|x-y\|;
$$
\item firmly nonexpansive if
$$
(\forall x \in X)(\forall y \in X) \quad
\|T x-T y\|^2+\|(\Id-T)x-(\Id-T)y\|^2\leqslant \|x-y\|^2;
$$
\item $\mu$-cocoercive if $\mu T$ is firmly nonexpansive.
\end{enumerate}
\end{definition}

\begin{definition}\label{d:averaged:his}
Let $T: X \rightarrow X$ be nonexpansive. $T$ is $k$-averaged\footnote{Usually,
one excludes the cases $k=0$ and $k=1$ in the study of averaged operators, but it is very convenient in this
paper to allow these cases.}
if $T$ can be represented as $$T=(1-k) \mathrm{Id}+k N,$$ where $N:X\rightarrow X$ is nonexpansive,
and $k \in[0,1]$.
\end{definition}

Averaged operators are important in optimization,
see, e.g., \cite{baillon, bartz, BBM,  BC, BM, bmwang, cegielski, comb04,  CY, OY,xu}.
Firmly nonexpansive operators, being $1/2$-averaged \cite[Proposition 4.4]{BC}, form a proper
subclass of the class of averaged operators.
From the definition we have $\Id$ is the only 0-averaged operator. When $k \in(0,1]$, various
characterizations of $k$-averagedness (see \cite[Proposition 2.2]{bartz}, \cite{BC, cegielski}) are available including
\begin{equation}\label{e:average1}
(\forall x \in X)(\forall y \in X) \quad\|T x-T y\|^2 \leqslant\|x-y\|^2-\frac{1-k}{k}\|(\operatorname{Id}-T) x-(\operatorname{Id}-T) y\|^2,
\end{equation}
and $(\forall x \in X)(\forall y \in X)$
\begin{equation}\label{e:average2}
\|T x-T y\|^2\leqslant \langle x-y, T x-T y\rangle +(1-2 k)(\scal{x-y}{Tx-Ty}-\|x-y\|^2 ).
\end{equation}

When $k=0$, while the historic Definition~\ref{d:averaged:his} gives
$T=\Id$ (linear), characterization~\eqref{e:average2} gives $T=\Id+v$ (affine) for some $v\in X$,
hence they are not equivalent in this case.
From \eqref{e:average1} or \eqref{e:average2} and the fact that $\mathrm{Id}$ is the only $0$-averaged operator, we can deduce that if an
operator is $k_0$-averaged, then it is $k$-averaged for every $k\geqslant k_0$.
This motivates the following definition, which was proposed by Bauschke, Bendit and Moursi \cite{BBM}.

\begin{definition}[Bauschke-Bendit-Moursi modulus of averagedness]\label{d:bbw:mod}
Let $T: X \rightarrow X$ be nonexpansive. The Bauschke-Bendit-Moursi modulus of averagedness of
$T$ is defined by
$$
k(T):=\inf \{k \in[0,1] \mid T \text { is } k \text {-averaged }\} .
$$
We call it the BBM modulus of averagedness.
\end{definition}

It is natural to ask: how does the modulus of averagedness impact classifications of averaged operators?
In view of Definition~\ref{d:bbw:mod},
if $T: X \rightarrow X$ is firmly nonexpansive then $k(T) \leqslant 1/2$.
Based on this, we define the following, which classifies the class of firmly nonexpansive operators
using the BBM modulus of averagedness.

\begin{definition}[normal and special nonexpansiveness]
Let $T: X \rightarrow X$. We say that $T$ is normally (firmly) nonexpansive if $k(T)<1/2$, and $T$ is specially (firmly) nonexpansive if $k(T)=1/2$.
\end{definition}

\begin{center}
\begin{tikzpicture}[x=0.75pt,y=0.75pt,yscale=-1,xscale=1]

\draw[black, line width=0.8pt]
(235.76,141.65) .. controls (235.76,88.38) and (278.94,45.2) .. (332.21,45.2)
.. controls (385.47,45.2) and (428.65,88.38) .. (428.65,141.65)
.. controls (428.65,194.91) and (385.47,238.09) .. (332.21,238.09)
.. controls (278.94,238.09) and (235.76,194.91) .. (235.76,141.65) -- cycle ;

\filldraw[fill=blue!8, draw=black, dashed, line width=0.8pt]
(284.39,141.65) .. controls (284.39,115.24) and (305.8,93.83) .. (332.21,93.83)
.. controls (358.62,93.83) and (380.03,115.24) .. (380.03,141.65)
.. controls (380.03,168.06) and (358.62,189.47) .. (332.21,189.47)
.. controls (305.8,189.47) and (284.39,168.06) .. (284.39,141.65) -- cycle ;

\draw (360,85.4) node [anchor=north west][inner sep=0.75pt, font=\footnotesize]  {$k=1/2$};
\draw (309.12,130.55) node [anchor=north west][inner sep=0.75pt, font=\footnotesize]  {$k< 1/2$};

\end{tikzpicture}
\end{center}

Let $\Gamma_0(X)$ denote
the set of all proper lower semicontinuous convex functions from $X$ to $(-\infty,+\infty]$.
Recall that for $f \in \Gamma_0(X)$ its proximal operator is defined by
$(\forall x\in X)\ \mathrm{P}_{f}(x):=\underset{u \in X}{\operatorname{argmin}}\left\{f(u)+\frac{1}{2}\|u-x\|^2\right\}$.
For a nonempty closed convex subset $C$ of $X$, its indicator function is defined by
$\iota_{C}(x):=0$ if $x\in C$, and $+\infty$ otherwise.  If $f=\iota_{C}$, we write
$\mathrm{P}_{f}=P_{C}$, the projection operator onto $C$.
It is well known that $\mathrm{P}_{f}$ is firmly nonexpansive \cite{BC}, which implies $k(\mathrm{P}_{f}) \leqslant 1/2$. Some natural questions arise: Given $f \in \Gamma_0(X)$, when is $\mathrm{P}_{f}$ normally (or specially) nonexpansive? how can we calculate
$k(\mathrm{P}_{f})$? In \cite{BBM}, these problems are essentially solved in
linear cases, or, in smooth case on the real line.

\emph{The goal of this paper is to classify averaged nonexpansive operators, including firmly nonexpansive operators, via the Bauschke-Bendit-Moursi
modulus of averagedness in a general Hilbert space. We provide some fundamental properties of modulus of averagedness
of averaged mappings, firmly nonexpansive
mappings and proximal mappings. We determine what properties normally (or specially) nonexpansive operators possess by using the monotone operator theory.
One amazing result is that a proximal mapping of a convex function has
its modulus of averagedness less than $1/2$ if and only if the function is Lipschitz smooth.
Many examples are provided to illustrate our results.
Bauschke-Bendit-Moursi
modulus of averagedness turns out to be an
extremely powerful tool in studying averaged operators and firmly nonexpansive operators!}

The rest of the paper is organized as follows. In Section~\ref{s:bilip} we explore some basic
properties of the modulus function and show that a normally nonexpansive operator is a bi-Lipschitz homeomorphism. In Section~\ref{s:app} averagedness of operator compositions and some asymptotic behaviors of averaged operators
are examined. In particular, the limiting operator of an averaged operator
is a projection if and only if its BBM modulus is $1/2$.
In Sections~\ref{s:resolvent} and \ref{s:proximal} we investigate both 
normal and special nonexpansiveness of resolvents and proximal operators. Our surprising results 
are Theorem~\ref{Charac1} and Theorem~\ref{t:lipsmooth}, characterizing normal and special resolvents and proximal operators. In Section~\ref{s:compute}
we establish formulae of modulus of averagedness of resolvents in terms of various values of maximally
monotone operators. Finally, in Section~\ref{s:twod} we extend a modulus of averagedness
formula on a composition of two projections by Bauschke, Bendit and Moursi in $\RR^2$ to general Hilbert spaces.

\section{Bijective theorem}\label{s:bilip}

\subsection{Auxiliary results}
This subsection collects preparatory results on modulus of averagedness used in later proofs.
For any operator $T: X \rightarrow X$ and any $v \in X$, the operator $T+v$ is defined by
$$
(\forall x \in X) \quad(T+v) x:=T x+v.
$$

\begin{proposition}\label{+v}
Let $T: X \rightarrow X$ be nonexpansive and $v \in X$. Then
\begin{enumerate}
\item\label{i:shift1}
$k(T+v)=k(T)$.
\item\label{i:shift2}
$k(T(\cdot+v))=k(T)$.
\end{enumerate}
\end{proposition}
\begin{proof}
\ref{i:shift1}: The result follows by combining $(T+v) x-(T+v) y=T x-T y$ with characterization (2).

\ref{i:shift2}: The result follows by combining $x-y=(x+v)-(y+v)$ with characterization (2).
\end{proof}

\begin{proposition}\label{attain}
Let $T: X \rightarrow X$ be nonexpansive. If $k(T)>0$, then $T$ is $k(T)$-averaged. Moreover,
$T$ is $\beta$-averaged for every $\beta\in [k(T),1]$.
\end{proposition}
\begin{proof}
Due to $k(T)>0$, we can use characterization either \eqref{e:average1} or \eqref{e:average2}.
The right hand side of \eqref{e:average1} or \eqref{e:average2} is a continuous
and increasing function in term of $k$, thus the result follows.
\end{proof}

Let $\Fix T:=\{x \in X \mid T x=x\}$ denote the set of fixed points of $T: X \rightarrow X$. Our following result characterizes $k(T)=0$.
\begin{proposition}\label{zero}
Let $T: X \rightarrow X$ be nonexpansive. Then
\begin{equation}\label{e:shiftonly}
k(T)=0 \Leftrightarrow \exists v \in X: T=\Id+v.
\end{equation}
If, in addition, $\Fix T\neq \varnothing$, then
\begin{equation}\label{e:identityonly}
k(T)=0 \Leftrightarrow T=\Id.
\end{equation}
\end{proposition}
\begin{proof}
Suppose $\exists v \in X: T=\mathrm{Id}+v$. Obviously $k(\Id)=0$. Thus by Proposition \ref{+v}, $k(T)=k(\mathrm{Id}+v)=0$.

Suppose $k(T)=0$. Assume that for any $v \in X$: $T \neq \Id+v$. Then there exist $x_0, y_0 \in X$
such that $(T-\mathrm{Id}) x_0 \neq(T-\mathrm{Id}) y_0$,
whence $\left\|(T-\mathrm{Id}) x_0-(T-\mathrm{Id}) y_0\right\|^2>0$.
Our assumption implies $T \neq \Id$, and $\Id$ is the only 0-averaged operator, thus there exists a sequence $(k_n)_{n \in \NN}$ in $(0,1]$ such that $T$ is $k_n$-averaged and $k_n \rightarrow 0$.
Now characterization \eqref{e:average1} implies that for any $n \in \NN$:
$$
\|T x_0-T y_0\|^2 \leqslant\|x_0-y_0\|^2-\frac{1-k_n}{k_n}\|(\operatorname{Id}-T) x_0-(\operatorname{Id}-T) y_0\|^2,
$$
i.e.,$$
0 \leqslant \left\|x_0-y_0\right\|^2-\left\|T x_0-T y_0\right\|^2 +(1-\frac{1}{k_n})\left\|(T-\Id) x_0-(T-\Id) y_0\right\|^2.
$$
Note that $\left\|(T-\mathrm{Id}) x_0-(T-\mathrm{Id}) y_0\right\|^2>0$. Now letting $n \rightarrow \infty$ yields $
0 \leqslant-\infty,
$ which is a contradiction.

When $\Fix T\neq\varnothing$, \eqref{e:identityonly} follows from \eqref{e:shiftonly}.
\end{proof}

\begin{proposition}\label{p:firm:mod}
Let $T: X \rightarrow X$ be nonexpansive. Then $T$ is firmly nonexpansive if and only if $
k(T)\leqslant 1/2$.
\end{proposition}
\begin{proof} ``$\Leftarrow$":
When $0<k(T)<1/2$,
apply Proposition~\ref{attain}. When $k(T)=0$, apply
Proposition~\ref{zero}.

``$\Rightarrow$": The assumption implies that $T$ is $1/2$-averaged. Hence $k(T)\leqslant 1/2$.
\end{proof}

\begin{example}\label{e:const} If $T:X\rightarrow X$ is a constant mapping, i.e., $(\exists v\in X)(\forall x\in X) \ Tx=v,$
then $k(T)=1/2$.
\end{example}
\begin{proof} Because $T$ is firmly nonexpansive, $k(T)\leqslant 1/2$.
By \eqref{e:average2}, if $T$ is $k$-averaged, then
$2k\geqslant 1$, so $k(T)\geqslant 1/2$. Altogether, $k(T)=1/2$.
\end{proof}

We end up this section with a fact on convexity.
\begin{fact}\emph{\cite[Fact 1.3]{BBM}}\label{f:bmoursi}
Let $T_{1}, T_{2}: X \rightarrow X$ be nonexpansive and $\lambda\in [0,1]$.
Then
$k(\lambda T_{1} +(1-\lambda)T_{2})\leqslant \lambda k(T_{1})+(1-\lambda)k(T_{2}).$
Consequently, $T\mapsto k(T)$ is a convex function
on the set of averaged mappings, as well as on the set of firmly nonexpansive mappings.
\end{fact}
\begin{corollary} Let $T: X \rightarrow X$ be nonexpansive and $\lambda\in [0,1]$.
Then $k(\lambda T)\leqslant \lambda k(T)+(1-\lambda)/2$.
\end{corollary}
\begin{proof}
Let $T_2$ be zero mapping in Fact~\ref{f:bmoursi} and apply Example~\ref{e:const}.
\end{proof}

\subsection{Bijective theorem}
In this subsection, we will show that normally nonexpansive operator must be bijective and bi-Lipschitz. Firstly we prove that normally nonexpansive operators must be bi-Lipschitz and injective.

\begin{lemma}\label{lemma}
Let $T: X \rightarrow X$ be normally nonexpansive. Then $T$ is a bi-Lipschitz homeomorphism from $X$ to $\operatorname{ran} T$. In particular, $T$ is injective.
\end{lemma}
\begin{proof}
In view of Proposition \ref{zero}, we may assume $k(T)>0$. Then $T$ is $k(T)$-averaged by Proposition~\ref{attain}, i.e., $$
(\forall x \in X)(\forall y \in X) \quad\|T x-T y\|^2+(1-2 k(T))\|x-y\|^2 \leqslant 2(1-k(T))\langle x-y, T x-T y\rangle.
$$

Since $k(T)<\frac{1}{2}$, there exists $\alpha \in (0,\frac{1}{2})$ such that $k(T)=\frac{1}{2}-\alpha$. Substituting $k(T)$ in above inequality and using the Cauchy-Schwarz inequality, we have
$$\begin{aligned}\left\|Tx-Ty\right\|^2+2 \alpha\|x-y\|^2 & \leqslant(1+2 \alpha)\langle x-y, T x-T y\rangle, \\ \left\|Tx-Ty\right\|^2+2 \alpha\|x-y\|^2 & \leqslant(1+2 \alpha)\left(\|x-y\|\left\|Tx-Ty\right\|\right), \\ 2 \alpha\left(\|x-y\|^2-\|x-y\|\left\|Tx-Ty\right\|\right) & \leqslant\|x-y\|\left\|Tx-Ty\right\|-\left\|Tx-Ty\right\|^2, \\ 2 \alpha\|x-y\|\left(\|x-y\|-\left\|Tx-Ty\right\|\right) & \leqslant\left\|Tx-Ty\right\|\left(\|x-y\|-\left\|Tx-Ty\right\|\right).\end{aligned}$$
Now if $\|x-y\|-\|T x-T y\|=0$, then $\left\|Tx-Ty\right\|=\|x-y\| \geqslant 2\alpha\|x-y\|$ since $2\alpha<1$. If $\|x-y\|-\left\|Tx-Ty\right\| \neq 0$, then $2 \alpha\|x-y\| \leqslant\|T x-T y\|$. Thus in both cases we have $2 \alpha\|x-y\| \leqslant\|T x-T y\|$. Combining it with $\|T x-T y\| \leqslant\|x-y\|$, we have $$
2 \alpha\|x-y\| \leqslant\left\|Tx-Ty\right\| \leqslant\|x-y\|.
$$
i.e., $T$ is a bi-Lipschitz homeomorphism from $X$ to $\operatorname{ran} T$.
\end{proof}

Next we make use of monotone operator theory to prove that normally nonexpansive operators must also be surjective.

\begin{fact}\label{firm}\emph{\cite[Example 20.30]{BC}}
Let $T: X \rightarrow X$ be firmly nonexpansive. Then $T$ is maximally monotone.  \end{fact}

\begin{fact}[Rockafellar–Vesely]\emph{\cite[Corollary 21.24]{BC}}\label{RV}
Let $A: X \rightrightarrows X$ be a maximally monotone operator such that $
\lim _{\|x\| \rightarrow+\infty} \inf \|A x\|=+\infty.
$
Then $A$ is surjective.
\end{fact}

\begin{lemma}\label{main1}
Let $T: X \rightarrow X$ be normally nonexpansive. Then $T$ is surjective.
\end{lemma}
\begin{proof}
By Lemma \ref{lemma}, $T$ is bi-Lipschitz since $T$ is normally nonexpansive. Thus there exists $\varepsilon>0$, such that $\varepsilon\|x-y\| \leqslant\|T x-T y\|$. Let $y=0$, then $\varepsilon\|x\| \leqslant\|T x-T 0\|$. Using the triangle inequality, we have
$$
\|T x\| \geqslant \varepsilon\|x\|-\|T 0\| .
$$Thus $\lim _{\|x\| \rightarrow \infty}\|T x\|=\infty$. Combining Fact \ref{firm} with Fact~\ref{RV} we complete the proof.
\end{proof}

\begin{theorem}[bi-Lipschitz homeomorphism]\label{bijectivethe}
Let $T: X \rightarrow X$ be normally nonexpansive. Then $T$ is a bi-Lipschitz homeomorphism of $X$. In particular, $T$ is bijective.
\end{theorem}

\begin{proof}
Combine Lemma \ref{lemma} and Lemma \ref{main1}.
\end{proof}

Taking the contrapositive of Theorem \ref{bijectivethe}, we obtain a lower bound for modulus of averagedness.
\begin{corollary}\label{Alowerbound}
Let $T: X \rightarrow X$ be nonexpansive. If $T$ is not bijective, then $k(T) \geqslant 1/2$.
\end{corollary}

\begin{remark}
In terms of compact operators (see, e.g., \cite{RU}), Theorem~\ref{bijectivethe} implies that $X$ is finite-dimensional if and only if there exists a normally nonexpansive compact operator on $X$.
\end{remark}

\begin{example}[averagedness of projection]\label{projectionaverage}
Let $C$ be a nonempty closed convex set in $X$ and $C \neq X$. Then $P_C$ is specially nonexpansive.
\end{example}

\begin{proof}
We have $k(P_C) \leqslant 1/2$ since $P_C$ is firmly nonexpansive. Now since $C \neq X$, let $x_0 \in X \backslash C$. Because $P_C\left(x_0\right) \in C$ and $x_0 \in X \backslash C$, we have $P_C\left(x_0\right) \neq x_0$. However, $P_C\left(x_0\right)=P_C\left(P_C\left(x_0\right)\right)$. Thus $P_C$ is not injective. Therefore $P_C$ is specially nonexpansive by Corollary \ref{Alowerbound}. Another way is to observe that $P_C$ is not surjective.
\end{proof}

\begin{corollary}
Let $M \in \mathbb{R}^{n \times n}$ be nonexpansive. If $\operatorname{det}(M)=0$, then $k(M) \geqslant 1 / 2$.
\end{corollary}

\begin{remark}
Consider the matrix $$
A=\frac{1}{2}\left(\begin{array}{cc}
2 & 0 \\
0 & -1
\end{array}\right).
$$Then one can verify that $k(A)=3/4>1/2$. However $\operatorname{det}(A) \neq 0$ and thus $A$ is a bi-Lipschitz homeomorphism of $\RR^{2}$. Hence the converse of Theorem~\ref{bijectivethe} fails. We will show later that the converse of Theorem~\ref{bijectivethe} does hold
when $T$ is a proximal operator; see Theorem~\ref{t:lipsmooth}.
\end{remark}

\section{Operator compositions and limiting operator}\label{s:app}
In this section, we examine the modulus of
averagedness of operator compositions and explore its asymptotic properties.

\subsection{Composition}
\begin{proposition}\label{Twoes}
Let $T_1$ and $T_2$ be nonexpansive operators from $X$ to $X$. Suppose one of the following holds:
\begin{enumerate}
\item $T_1$ is not surjective.
\item $T_2$ is not injective.
\item $T_1$ is bijective and $T_2$ is not surjective.
\item $T_2$ is bijective and $T_1$ is not injective.
\end{enumerate}
Then $k(T_1 T_2) \geqslant 1/2$.
\end{proposition}

\begin{proof}
Since $T_1$ and $T_2$ are nonexpansive operators, we have $T_1 T_2$ is nonexpansive as well. Each one of the four conditions implies that $T_1 T_2$ is not bijective. Now use Corollary \ref{Alowerbound}.
\end{proof}

Ogura and Yamada \cite{OY} obtained the following result about the
averagedness of operator compositions.

\begin{fact}\emph{\cite[Theorem 3]{OY} (see also \cite[Proposition 2.4]{CY})} \label{OYFACT}
Let $T_1: X \rightarrow X$ be $\alpha_1$-averaged, and let $T_2: X \rightarrow X$ be $\alpha_2$-averaged, where $\alpha_1, \alpha_2 \in (0,1)$. Set
$$
T=T_1 T_2 \quad \text { and } \quad \alpha=\frac{\alpha_1+\alpha_2-2 \alpha_1 \alpha_2}{1-\alpha_1 \alpha_2} .
$$
Then $\alpha \in( 0,1)$ and $T$ is $\alpha$-averaged.
\end{fact}

Formulating this result here using the modulus of
averagedness, we have the following result.
\begin{proposition}\label{OYPROP}
Let $T_1: X \rightarrow X$ and $T_2: X \rightarrow X$ be nonexpansive. Suppose $k\left(T_1\right)k\left(T_2\right) \neq 1$. Then
$$
k\left(T_1 T_2\right) \leqslant \frac{k\left(T_1\right)+k\left(T_2\right)-2 k\left(T_1\right) k\left(T_2\right)}{1-k\left(T_1\right) k\left(T_2\right)}.
$$
\end{proposition}
\begin{proof}
Let $\varphi\left(T_1, T_2\right):=\frac{k\left(T_1\right)+k\left(T_2\right)-2 k\left(T_1\right) k\left(T_2\right)}{1-k\left(T_1\right) k\left(T_2\right)}$. We consider five cases.

Case 1: $k\left(T_i\right)=1$ for some $i \in\{1,2\}$. Then $\varphi\left(T_1, T_2\right)=1$. Since $T_1$ and $T_2$ are nonexpansive, we have $T_1 T_2$ is nonexpansive, i.e., $k\left(T_1 T_2\right) \leqslant 1=\varphi\left(T_1, T_2\right)$.

Case 2: $k\left(T_i\right) \in(0,1)$ for any $i \in\{1,2\}$. Then combining Proposition \ref{attain} and Fact \ref{OYFACT} we have $T_1 T_2$ is $\varphi\left(T_1, T_2\right)$-averaged. Thus $k\left(T_1T_2\right) \leqslant \varphi\left(T_1, T_2\right)$.

Case 3: $k\left(T_1\right)=0$ and $k\left(T_2\right) \in(0,1)$. Then there exists $v_1 \in X$ such that $T_1=\Id+v_1$ by Proposition \ref{zero}. Thus $T_1 T_2=T_2+v_1$ and $k\left(T_1 T_2\right)=k\left(T_2+v_1\right)=k\left(T_2\right)$ by Proposition \ref{+v}. While $\varphi\left(T_1, T_2\right)=k\left(T_2\right)$ in this case, we have $k\left(T_1T_2\right)=\varphi\left(T_1, T_2\right)$.

Case 4: $k\left(T_1\right) \in(0,1)$ and $k\left(T_2\right)=0$. Then there exists $v_2 \in X$ such that $T_2=\Id+v_2$ by Proposition \ref{zero}. Thus $T_1 T_2=T_1(\cdot+v_2)$ and $k\left(T_1 T_2\right)=k\left(T_1(\cdot+v_2)\right)=k\left(T_1\right)$ by Proposition \ref{+v}. While $\varphi\left(T_1,T_2\right)=k\left(T_1\right)$ in this case, we have $k\left(T_1T_2\right)=\varphi\left(T_1, T_2\right)$.

Case 5: $k\left(T_1\right)=k\left(T_2\right)=0$. Then there exist $v_1 \in X$ and $v_2 \in X$ such that $T_1=\Id+v_1$ and $T_2=\Id+v_2$ by Proposition \ref{zero}. Thus $T_1 T_2=\mathrm{Id}+v_2+v_1$ and $k\left(T_1 T_2\right)=k\left(\mathrm{Id}+v_2+v_1\right)=0$. While $\varphi\left(T_1,T_2\right)=0$ in this case, we have $k\left(T_1T_2\right)=\varphi\left(T_1, T_2\right)$.

Altogether we complete the proof.
\end{proof}

\begin{proposition}
Let $C$ be a nonempty closed convex set in $X$ and $C \neq X$. Then for any nonexpansive operator $T: X \rightarrow X$:
$$
\frac{1}{2} \leqslant k\left(T \circ P_C\right) \leqslant \frac{1}{2-k(T)}
$$and
$$
\frac{1}{2} \leqslant k\left(P_C \circ T\right) \leqslant \frac{1}{2-k(T)} .
$$
\end{proposition}
\begin{proof}
Observe that $P_C$ is neither surjective nor injective in this case. Thus by Proposition \ref{Twoes} we have $k\left(T \circ P_C\right) \geqslant 1/2$ and $k\left(P_C \circ T\right) \geqslant 1/2$. Now by Example \ref{projectionaverage}$$
\frac{k\left(T\right)+k\left(P_C\right)-2 k\left(T\right) k\left(P_C\right)}{1-k\left(T\right) k\left(P_C\right)}=\frac{1}{2-k(T)}.
$$
Thus by Proposition \ref{OYPROP} we have $k\left(T \circ P_C\right) \leqslant \frac{1}{2-k(T)}$ and $k\left(P_C \circ T\right) \leqslant \frac{1}{2-k(T)}$, which complete the proof.
\end{proof}
\begin{remark}
Particularly, if we let $T=P_V$ and $C=U$, where $U$ and $V$ are both closed linear subspaces, then $k\left(P_V P_U\right)=\frac{1+c_F}{2+c_F} \in\left[\frac{1}{2}, \frac{2}{3}\right]$, where $c_F \in[0,1]$ (see \cite[Corollary 3.3]{BBM}). This coincides with the bounds we obtained as $\frac{1}{2-k(P_U)}=\frac{2}{3}$ by Example \ref{projectionaverage}.
\end{remark}

We can generalize the results of two operator compositions to finite operator compositions.

\begin{proposition}
Let $m \geqslant 2$ be an integer and let $I=\{1, \ldots, m\}$. For any $i \in I$, let $T_i$ be nonexpansive from $X$ to $X$. Suppose one of the following holds:
\begin{enumerate}
\item $T_1$ is not surjective.
\item $T_m$ is not injective.
\item $T_1$ is bijective and $T_2 \cdots T_m$ is not surjective.
\item $T_m$ is bijective and $T_1 \cdots T_{m-1}$ is not injective.
\end{enumerate}
Then $k\left(T_1 \cdots T_m\right) \geqslant 1/2$.
\end{proposition}

\begin{proof}
Apply Proposition \ref{Twoes}.
\end{proof}

\begin{corollary}
Let $C_1, \ldots, C_m$ be nonempty closed convex sets in $X$. If $C_1 \neq X$ or $C_m \neq X$, then $k\left(P_{C_1} \cdots P_{C_m}\right) \geqslant 1/2$.
\end{corollary}

The following result is about modulus of averagedness of isometries.
\begin{proposition}
Let $A$ be a $n \times n$ orthogonal matrix and $A \neq \Id$. Then $k(A)=1$. \end{proposition}

\begin{proof}
Since $A$ is orthogonal, we have  $\left\|Ax-Ay\right\|=\|x-y\|$. On the other hand, $\operatorname{ran}(\operatorname{Id}-A)$ is not a singleton. Hence $k(A)=1$ by using \eqref{e:average1}.
\end{proof}

\begin{corollary}\label{ortcompo}
Let $m \geqslant 1$ be an integer and let $I=\{1, \ldots, m\}$. For any $i \in I$, let $A_i$ be a $n \times n$ orthogonal matrix. Suppose that $A_1 \cdots A_m \neq \Id$. Then $k(A_1 \cdots A_m)=1$.
\end{corollary}

\subsection{Limiting operator}

In this subsection we discuss the asymptotic behavior of modulus of averagedness. Recall that a sequence $\left(x_n\right)_{n \in \mathbb{N}}$ in a Hilbert space $X$ is said to converge weakly to a point $x$ in $X$ if
$(\forall y\in X)\ \lim _{n \rightarrow \infty}\left\langle x_n, y\right\rangle=\langle x, y\rangle.
$
We use the notation $\lim _{n \rightarrow \infty}^w x_n$ for the weak limit of
$\left(x_n\right)_{n \in \mathbb{N}}$. Recall for a nonexpansive operator $T: X \rightarrow X$, $\Fix T$ is closed and convex, see, e.g., \cite[Proposition 22.9]{BM}.

\begin{fact}\emph{\cite[Proposition 5.16]{BC}} \label{avconverge}
Let $\alpha \in (0,1)$ and let $T: X \rightarrow X$ be $\alpha$-averaged such that $\operatorname{Fix} T \neq \varnothing$. Then for any $x \in X$, $\left(T^n x\right)_{n \in \mathbb{N}}$ converges weakly to a point in $\operatorname{Fix} T$.
\end{fact}

In view of the above fact, we propose the following type of operator.
\begin{definition}[limiting operator]
Let $\alpha \in (0,1)$ and let $T: X \rightarrow X$ be $\alpha$-averaged such that $\operatorname{Fix} T \neq \varnothing$. Define its limiting operator $T_{\infty}: X \rightarrow X$ by $x\mapsto \lim^w_{n\rightarrow\infty} T^n x.$
\end{definition}

\begin{remark}
The full domain and single-valuedness of $T_{\infty}$ are guaranteed by Fact~\ref{avconverge}.
Hence $T_{\infty}: X \rightarrow X$ is well defined.
\end{remark}

\begin{example}\label{linearTinfi}
\begin{enumerate}
\item \cite[Example 5.29]{BC}
Let $\alpha \in (0,1)$ and let $T: X \rightarrow X$ be $\alpha$-averaged such that $\Fix T \neq \varnothing$. Suppose $T$ is linear. Then $T_{\infty}=P_{\Fix T}$.
\item \cite[Proposition 5.9]{BC}
Let $\alpha \in (0,1)$ and let $T: X \rightarrow X$ be $\alpha$-averaged such that $\Fix T \neq \varnothing$. Suppose $\Fix T$ is a closed affine subspace of $X$. Then $T_{\infty}=P_{\Fix T}$.
\end{enumerate}
\end{example}

The limiting operator of an averaged mapping enjoys the following pleasing properties.
\begin{proposition}\label{limiting}
Let $\alpha \in (0,1)$ and let $T: X \rightarrow X$ be $\alpha$-averaged such that $\Fix T \neq \varnothing, X$. Then the following hold:
\begin{enumerate}
\item \label{i:lim1} $\Fix T=\Fix T_{\infty}=\operatorname{ran} T_{\infty}$.

\item \label{i:lim2} $\left(T_{\infty}\right)^2=T_{\infty}$.

\item \label{i:lim3} $k\left(T_{\infty}\right) \in\left[\frac{1}{2}, 1\right]$.
\end{enumerate}
\end{proposition}

\begin{proof}
\ref{i:lim1}  If $x \notin \Fix T$, then $T_{\infty}x \neq x$ since $T_{\infty}x \in \Fix T$ by
Fact~\ref{avconverge}. If $x \in \Fix T$, then $T_{\infty}x=\lim^w _{n \rightarrow \infty} T^n x=\lim^w _{n \rightarrow \infty} x=x$. Thus $\operatorname{Fix} T=\operatorname{Fix} T_{\infty}$. The equality $\Fix T=\operatorname{ran} T_{\infty}$ follows by using Fact \ref{avconverge} again.

\ref{i:lim2} For any $x \in X$, $T_{\infty} x \in \operatorname{ran} T_{\infty}$, thus $T_{\infty} x \in \Fix T_{\infty}$ by \ref{i:lim1}. Therefore $\left(T_{\infty}\right)^2 x=T_{\infty}\left(T_{\infty} x\right)=T_{\infty} x$, which implies that $\left(T_{\infty}\right)^2=T_{\infty}$.

\ref{i:lim3}
Since the norm is weakly lower-semicontinuous, we have $$
(\forall x \in X)(\forall y \in X) \quad\left\|T_{\infty} x-T_{\infty} y\right\| \leqslant \liminf _{n \rightarrow \infty} \left\|T^n x-T^n y\right\|.
$$
As $T: X \rightarrow X$ is nonexpansive, by induction we have for any $n \in \mathbb{N}$, $\left\|T^n x-T^n y\right\| \leqslant\|x-y\|$. Altogether, $T_{\infty}$ is nonexpansive, which implies that $k(T_{\infty}) \leqslant 1$. On the other hand, we have $\operatorname{ran} T_{\infty} = \operatorname{Fix} T \neq X$ by \ref{i:lim1} and the
assumption, so $T_{\infty}$ is not surjective.
Thus $k(T_{\infty}) \geqslant 1/2$ by Corollary \ref{Alowerbound}.
\end{proof}

The BBM modulus of averagedness provides further insights into the limiting operator.
\begin{theorem}\label{limitingtheorem}
Let $\alpha \in (0,1)$ and let $T: X \rightarrow X$ be $\alpha$-averaged such that $\Fix T \neq \varnothing, X$. Then the following are equivalent:
\begin{enumerate}
\item \label{i:limtheo1} $T_{\infty}=P_{\Fix T}$.

\item \label{i:limtheo2} $k\left(T_{\infty}\right) \leqslant 1/2$.

\item \label{i:limtheo3} $k\left(T_{\infty}\right)=1/2$.
\end{enumerate}
\end{theorem}

\begin{proof}
\ref{i:limtheo1} $\Rightarrow$ \ref{i:limtheo2}: Obvious.

\ref{i:limtheo2} $\Rightarrow$ \ref{i:limtheo1}: The result follows by combining
Proposition~\ref{limiting}\ref{i:lim1}\&\ref{i:lim2} and the fact that if $T: X \rightarrow X$ is firmly nonexpansive and $T \circ T=T$, then $T=P_{\ran T}$ (see \cite[Exercise 22.5]{BM} or \cite[Theorem 2.1(xx)]{moffat}).

\ref{i:limtheo2} $\Leftrightarrow$ \ref{i:limtheo3}: Apply Proposition~\ref{limiting}\ref{i:lim3}.
\end{proof}

In the following, we discuss limiting operator on $\RR$. The following extends
\cite[Proposition 2.8]{BBM} from differentiable functions to locally Lipschitz functions.
Below $\partial_{L}g$ denotes the Mordukhovich limiting subdifferential \cite{zhu,boris-book,RW}.

\begin{lemma}\label{modulusR}
Let $g: \mathbb{R} \rightarrow \mathbb{R}$ be a locally Lipschitz function. Then
$g$ is nonexpansive
if and only if $(\forall x\in\RR)\
\partial_{L}g(x)\subset [-1,1]$ in which case $k(g)=\left(1-\inf \partial_{L}g(\mathbb{R})\right) / 2$.
\end{lemma}
\begin{proof} The nonexpansiveness characterization of $g$ follows from \cite[Theorem 3.4.8]{zhu}. Write $g=(1-\alpha)\Id+\alpha N$ where $\alpha\in [0,1]$ and $N:\RR\rightarrow\RR$ is nonexpansive.
If $\alpha=0$, the result clearly holds. Let us assume $\alpha>0$.
Then $N(x)=(g(x)-(1-\alpha)x)/\alpha$ and $\partial_{L}N(x)=
(\partial_{L}g(x)-(1-\alpha))/\alpha$. $N$ is nonexpansive is equivalent to
$$(\forall x\in\RR)\ (\partial_{L}g(x)-(1-\alpha))/\alpha\subseteq [-1,1]\quad \Leftrightarrow\quad (\forall x\in\RR)\ \partial_{L} g(x)\subseteq
[1-2\alpha,1],$$
from which
$$\alpha\geqslant \frac{1-\inf \partial_{L}g(\RR)}{2}$$
and the result follows.
\end{proof}

In Example \ref{linearTinfi} we see that if $T: X \rightarrow X$ is $\alpha$-averaged and linear with $\operatorname{Fix} T \neq \varnothing, X$, then $k\left(T_{\infty}\right)=1/2$. The following example shows that
it is not true in nonlinear case.

\begin{example}
Let
$$f(x):=\begin{cases}
0 & \text{if $x \leqslant 0$,} \\
x & \text{if $0 \leqslant x \leqslant 1$,}\\
-\frac{1}{2} x+\frac{3}{2} & \text{if $x \geqslant 1.$}
 \end{cases}
$$
Then $f$ is $(3/4)$-averaged and $\operatorname{Fix} T=[0,1]$. However,
 $$f_{\infty}(x)=\begin{cases}
 0 & \text{if $x \leqslant 0$ or  $x \geqslant 3$,}\\
 x & \text{if $0 \leqslant x \leqslant 1$,}\\
 -\frac{1}{2} x+\frac{3}{2} & \text{if $1 \leqslant x \leqslant 3,$}
   \end{cases}
 $$
 and $k\left(f_{\infty}\right)=3/4$.
\end{example}
\begin{proof}
By computation we have
$$\partial_{L}f(x)=\begin{cases}
\{0\} & \text{if $x<0$,}\\
[0,1] & \text{if $x=0$,}\\
\{1\} & \text{if $0<x<1$,}\\
\{-1/2, 1\} & \text{if $x=1$,}\\
\{-1/2\} & \text{if $x>1$,}
\end{cases}
\quad \text{and }
\partial_{L}f_{\infty}(x)=\begin{cases}
\{0\} & \text{if $x<0$ or $x>3$,}\\
[0,1] & \text{if $x=0$,}\\
\{1\} & \text{if $0<x<1$,}\\
\{-1/2, 1\} & \text{if $x=1$,}\\
\{-1/2\} & \text{if $1<x<3$,}\\
[-1/2,0] & \text{if $x=3$.}
\end{cases}
$$
Applying Lemma~\ref{modulusR} we obtain $k(f)$ and $k(f_{\infty})$.
\end{proof}

Next, we show that if $T: X \rightarrow X$ is firmly nonexpansive, a stronger condition than averagedness, then on the real line it is true that $k\left(T_{\infty}\right)=1/2$.

\begin{proposition}\label{firmlyR}
Let $f: \RR \rightarrow \RR$ be firmly nonexpansive such that $\Fix f \neq \varnothing, \RR$. Then $f_{\infty}=P_{\Fix f}$. Consequently, $k\left(f_{\infty}\right)=1 / 2$.
\end{proposition}
\begin{proof}
Since $f$ is firmly nonexpansive, we have $f$ is nondecreasing and nonexpansive. Now as $\Fix f\subseteq\RR$ is closed and convex, it must be one of the form
$[a,+\infty)$, $(-\infty,b]$ or $[a,b]$ with $a,b\in \RR$ because $\operatorname{Fix} f \neq \varnothing, \RR$.
Since the proofs for all cases are similar, let us assume that $\Fix f=[a,b]$.
When $x\geqslant b$, because $f$ is nondecreasing,
we have
$f(x)\geqslant f(b)=b$, $f^2(x)\geqslant f(b)=b$, and an induction leads
$f^n(x)\geqslant b$.  Then $f_{\infty}(x)\geqslant b$ by
Fact~\ref{avconverge}. Since $f_{\infty}(x)\in [a,b]$ by Fact~\ref{avconverge} again,
we derive that $f_{\infty}(x)=b$. Similar arguments give $f_{\infty}(x)=a$ when
$x\leqslant a$. Clearly,  when $x\in [a,b]$, $(\forall n\in\NN)\ f^{n}(x)=x$,
so $f_{\infty}(x)=x$.
Altogether $f_{\infty}=P_{\Fix f}$.
\end{proof}


Motivated by Example \ref{linearTinfi} and Proposition \ref{firmlyR} one might conjecture that $k\left(T_{\infty}\right)=1 / 2$ whenever $k\left(T\right) \leqslant 1 / 2$. However, this is not true in general.
To find a counter example,
by Theorem \ref{limitingtheorem}, it suffices to find a firmly nonexpansive operator such that its limiting operator is not a projection. We conclude this section with the following example from \cite[Example 4.2]{BDNP}.

\begin{example}\label{lim:counterexample}
Suppose that $X=\mathbb{R}^2$. Let $A=\RR (1, 1)$ and $B=\left\{(x, y) \in \mathbb{R}^2 \mid-y \leqslant x \leqslant 2\right\}$.
For $z=(x, y)\in\RR^2$ we have
$P_A(z)=\left(\frac{x+y}{2}, \frac{x+y}{2}\right)$ and
$$P_{B}(z)=\begin{cases}
(2, y) & \text{if $x\geq 2, y\geq -2$,}\\
(2,-2) & \text{if $y\leq\min\{x-4, -2\}$,}\\
((x-y)/2, -(x-y)/2) & \text{if $x-4<y\leq -x$,}\\
(x,y) & \text{if $(x,y)\in B$.}
\end{cases}
$$
Then the  \emph{Douglas–Rachford} operator $T=\operatorname{Id}-P_A+P_B (2 P_A-\mathrm{Id})$ is firmly nonexpansive and
has $k(T_{\infty})>1/2$. By Theorem \ref{limitingtheorem}, it suffices to show $T_{\infty} \neq P_{\Fix T}$.
Indeed,
by \cite[Fact 3.1]{BDNP} we have $\Fix T=\menge{s(1,1)}{s\in [0,2]}$ because of $A\cap\inte B\neq\varnothing$.
Let $z_0=(4,10)$, and $(\forall n \in \mathbb{N})$ $z_{n+1}=T z_n$.  Direct computations give
 $$z_1=(-1, 7), z_2=(-2, 3), z_3=(-1/2, 1/2), \text{ and } z_4=(0,0).$$
On the other hand, let $z^*=(2, 2)$, then $\{z_4, z^*\} \subset \Fix T$. Thus $T_{\infty}z_0=z_4$ while $P_{\Fix T}z_0 \neq z_4$ as $\left\|z_0-z^*\right\|=2 \sqrt{17} < \left\|z_0-z_4\right\|=2 \sqrt{29}$.
\end{example}

\section{Resolvent}\label{s:resolvent}
Let $A: X \rightrightarrows X$ be a set-valued operator, i.e., a mapping from $X$ to its power set. Recall that the resolvent of $A$ is $J_A :=(\mathrm{Id}+A)^{-1}$ and the reflected resolvent of $A$ is $R_A :=2 J_A-\mathrm{Id}$. The graph of $A$ is $
\gra A :=\{(x, u) \in X \times X \mid u \in A x\}$ and the inverse of $A$, denoted by $A^{-1}$, is the operator with graph $\operatorname{gra} A^{-1} :=\{(u, x) \in X \times X \mid u \in A x\}$. The domain of $A$ is $\operatorname{dom} A :=\{x \in X \mid A x \neq \varnothing\}$. $A$ is monotone, if
$$
\forall(x, u),(y, v) \in \operatorname{gra} A, \quad\langle x-y, u-v\rangle \geqslant 0.
$$
$A$ is maximally monotone, if it is monotone and there is no monotone operator $B: X \rightrightarrows X$ such that $\operatorname{gra} A$ is properly contained in $\operatorname{gra} B$.
Unless stated otherwise, we assume from now on that
$$
A: X \rightrightarrows X \text { and } B: X \rightrightarrows X \text { are maximally monotone operators.}
$$

\begin{fact}[Minty's theorem]\label{Minty}\emph{\cite[Proposition 23.8]{BC}}
Let $T: X \rightarrow X$. Then $T$ is firmly nonexpansive if and only if $T$ is the resolvent of a maximally monotone operator.
\end{fact}

The goal of this section is to give characterizations of
normal and special nonexpansiveness by using the monotone operator theory.

\subsection{Auxiliary results}
We first provide a nice formula for the modulus of averagedness of
$(1-\lambda)\Id+\lambda T$ in terms of the modulus of averagedness of $T$.
The following is an adaption of \cite[Proposition 4.40]{BC}. For completeness, we include a simple proof.

\begin{fact}\label{f:bc-ave} Let $T:X\rightarrow X$ be nonexpansive and let $\lambda\in (0, 1]$.
For $\alpha\in [0, 1]$, $T$ is $\alpha$-averaged if and only if $(1-\lambda)\Id+
\lambda T$ is $\lambda\alpha$-averaged.
\end{fact}
\begin{proof}
Suppose $T$ is $\alpha$-averaged. Then
$T=(1-\alpha)\Id+\alpha R$ with $R$ being nonexpansive. It follows that
\begin{align}
(1-\lambda)\Id +\lambda T &=(1-\lambda)\Id +\lambda(1-\alpha)\Id+\lambda\alpha R\\
&=(1-\lambda\alpha)\Id+\lambda\alpha R,
\end{align}
so that $(1-\lambda)\Id +\lambda T$ is $\lambda\alpha$-averaged. Because $\lambda\in (0,1]$,
the reverse direction also holds.
\end{proof}

\begin{lemma}\label{multi}
Let $T: X \rightarrow X$ be nonexpansive. Then for every $\lambda \in[0,1]$ we have
\begin{equation}\label{e:convex:comb1}
k((1-\lambda) \mathrm{Id}+\lambda T)=\lambda k(T).
\end{equation}
\end{lemma}
\begin{proof} We split the proof into two cases.

Case 1: $\lambda=0$. Clearly \eqref{e:convex:comb1} holds  because $k(\Id)=0$.

Case 2: $\lambda>0$.  We show \eqref{e:convex:comb1} by two subcases.

Subcase 1: $k((1-\lambda) \mathrm{Id}+\lambda T)=0$. By Proposition \ref{zero}, there exists $v \in X$: $
(1-\lambda) \mathrm{Id}+\lambda T=\Id+v$ such that $T=\Id+v/\lambda$.
Then $k((1-\lambda) \Id+\lambda T)=0=k(T)$ by Proposition \ref{zero} again.

Subcase 2: $k((1-\lambda) \mathrm{Id}+\lambda T)>0$. On one hand, we derive
$k((1-\lambda)\Id+\lambda T)\leqslant \lambda k(T)$ by Fact~\ref{f:bc-ave}.
On the other hand, since $(1-\lambda) \mathrm{Id}+\lambda T$ is $\lambda$-averaged, we have
$0<k((1-\lambda) \mathrm{Id}+\lambda T)\leqslant \lambda$. For every $\beta\in [k((1-\lambda) \mathrm{Id}+\lambda T),
\lambda]$, the mapping $(1-\lambda) \mathrm{Id}+\lambda T$ is $\beta$-averaged. Write $\beta=\lambda\alpha$
with $\alpha=\beta/\lambda\in (0,1]$. Fact~\ref{f:bc-ave} implies that
$T$ is $\alpha$-averaged, thus $k(T)\leqslant \beta/\lambda$. Taking infimum over $\beta$ gives
$k(T)\leqslant k((1-\lambda)\Id+\lambda T)/\lambda$, i.e.,  $\lambda k(T)\leqslant k((1-\lambda)\Id+\lambda T)$.
Therefore, $k((1-\lambda)\Id+\lambda T)=\lambda k(T)$.

Altogether, \eqref{e:convex:comb1} holds.
\end{proof}

\begin{example}
Let $C$ be a nonempty closed convex set in $X$ and $C \neq X$. Consider the reflector to $C$ defined by $R_C:=2 P_C-\mathrm{Id}$. Then the following hold:
\begin{enumerate}
\item\label{i:proj} $k(R_{C})=1$.
\item\label{i:refl} For $\lambda\in [0,1]$, $k((1-\lambda)\Id+\lambda R_{C}))=\lambda$.
\item \label{i:projr} For $\lambda\in [0,1]$, $k((1-\lambda)\Id+\lambda P_{C}))=\lambda/2$.
\end{enumerate}
\end{example}

\begin{proof}
Apply Example \ref{projectionaverage} and Lemma \ref{multi}.
\end{proof}

\begin{remark} This recovers \cite[Example 2.3]{BBM} for $C=V$,
a closed subspace of $X$.
\end{remark}

\begin{example}
Let $A: X \rightrightarrows X$ be maximally monotone. Consider the reflected resolvent of $A$ defined
by $R_{A}:=2J_{A}-\Id$. Then $k\left(R_A\right)=2k\left(J_A\right)$ by Lemma~\ref{multi}. Consequently,  $k\left(R_A\right)<1$ (that is, $R_A$ is $\alpha$-averaged for some $\alpha \in [0, 1)$) if and only if $J_A$ is normally nonexpansive. Likewise, $k\left(R_A\right)=1$ if and only if $J_A$ is specially nonexpansive.
\end{example}

The following result concerning the Douglas-Rachford operator (see, e.g., \cite{BC, BM}) is of independent interest.

\begin{theorem}
Let $U, V$ be two closed subspaces of $X$, and $U\neq V$. Consider the Douglas-Rachford operator
$$T_{U,V} :=\frac{\Id+R_{U}R_{V}}{2}.$$
Then $k(T_{U,V})=1/2$.
\end{theorem}
\begin{proof}
We have $R_U R_V \neq \Id$ since $U \neq V$. Note both $R_U$ and $R_V$ are orthogonal. Thus $k\left(R_U R_V\right)=1$ by Corollary~\ref{ortcompo}. Therefore, by Lemma~\ref{multi}, we have $k(T_{U,V})=k(R_{U}R_{V})/2=1/2.$
\end{proof}
\begin{remark} Let $A,B:X\rightrightarrows X$ be two maximally monotone operators. The Douglas-Rachford operator related to $(A,B)$ is
$$T_{A,B}=\frac{\Id+R_{A}R_{B}}{2}.$$
It is interesting to know $k(T_{A,B})$ in general.
\end{remark}

Next, we recall Yosida regularizations of monotone operators. They are essential for our proofs in Section~\ref{s:charc}.

\begin{definition}[Yosida regularization] For $\mu>0$, the Yosida $\mu$-regularization of $A$ is the operator
$$
Y_\mu(A):=\left(\mu \Id+A^{-1}\right)^{-1}.
$$
\end{definition}

For Yosida regularization, we have the classic identity: $Y_\mu(A)=\mu^{-1}\left(\Id-J_{\mu A}\right),$ see \cite[Lemma 12.14]{RW}. The following result is \cite[Theorem 23.7(iv)]{BC}. Here we take the opportunity to give a
detailed proof.

\begin{proposition}\label{Yosida1}
For $\alpha, \mu>0$, the following formula holds
$$
J_{\alpha Y_\mu(A)}=\frac{\mu}{\mu+\alpha} \Id+\frac{\alpha}{\mu+\alpha} J_{(\mu+\alpha) A}.
$$
\end{proposition}

\begin{proof}
Firstly,$$
\begin{aligned}
\alpha Y_\mu(A) & =\alpha\left(\mu \Id+A^{-1}\right)^{-1}=\left[\left(\mu\Id+A^{-1}\right) (\alpha^{-1}\Id)\right]^{-1} \\
& =\left(\alpha^{-1} \mu \Id+A^{-1} (\alpha^{-1} \Id)\right)^{-1}=\left(\alpha^{-1} \mu \Id+(\alpha A)^{-1}\right)^{-1} \\
& =Y_{\alpha^{-1} \mu}(\alpha A).
\end{aligned}
$$
Thus we only need to prove the formula holds for $\alpha=1$.

Let $y \in X$, $z=J_{(\mu+1) A}(y)$ and $x=\frac{\mu}{\mu+1} y+\frac{1}{\mu+1} z$. We will prove $x=J_{Y_{\mu(A)}}(y)$. We have $z=(\mu+1) x-\mu y$, $y-z=\frac{\mu+1}{\mu}(x-z)$ and $Y_\mu(A)=\frac{1}{\mu}\left(\Id-J_{\mu A}\right)$. Thus
$$\begin{aligned} z=J_{(\mu+1) A}(y) & \Leftrightarrow  y-z \in(\mu+1) A z  \Leftrightarrow \frac{\mu+1}{\mu}(x-z) \in(\mu+1) A z \\ & \Leftrightarrow x-z \in(\mu A) z  \Leftrightarrow z=J_{\mu A} x \Leftrightarrow(\mu+1) x-\mu y=J_{\mu A} x \\ & \Leftrightarrow y-x=\frac{x-J_{\mu A} x}{\mu}=Y_\mu(A)(x) \Leftrightarrow x=J_{Y_{\mu(A)}}(y).\end{aligned}
$$
\end{proof}

Combining Proposition \ref{Yosida1} and Lemma \ref{multi}, we have the following.

\begin{corollary}\label{Yosida2}
For any $\alpha \in[0,1)$, the following hold:
\begin{enumerate}
\item $J_{\alpha Y_{1-\alpha}(A)}=(1-\alpha) \mathrm{Id}+\alpha J_A$.
\item $k(J_{\alpha Y_{1-\alpha}}(A))=\alpha k(J_A)$.
\end{enumerate}
\end{corollary}

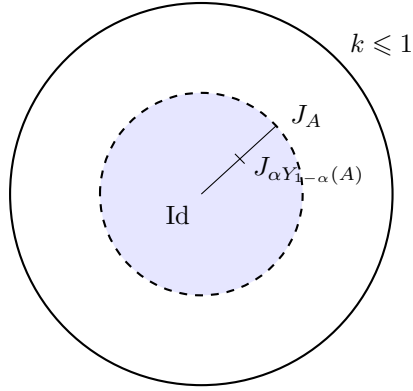
\begin{center}
\begin{tikzpicture}[x=0.75pt,y=0.75pt,yscale=-1,xscale=1]

\draw[black, line width=0.8pt]
(236.76,163.65) .. controls (236.76,110.38) and (279.94,67.2) .. (333.21,67.2)
.. controls (386.47,67.2) and (429.65,110.38) .. (429.65,163.65)
.. controls (429.65,216.91) and (386.47,260.09) .. (333.21,260.09)
.. controls (279.94,260.09) and (236.76,216.91) .. (236.76,163.65) -- cycle ;

\filldraw[fill=blue!10, draw=black, dashed, line width=0.8pt]
(282.11,163.65) .. controls (282.11,135.43) and (304.99,112.55) .. (333.21,112.55)
.. controls (361.43,112.55) and (384.31,135.43) .. (384.31,163.65)
.. controls (384.31,191.87) and (361.43,214.75) .. (333.21,214.75)
.. controls (304.99,214.75) and (282.11,191.87) .. (282.11,163.65) -- cycle ;

\draw    (371.12,129.02) -- (333.21,163.65) ;
\draw    (350.16,143.34) -- (355.05,148.26) ;

\draw (407,81.4) node [anchor=north west][inner sep=0.75pt, font=\footnotesize]  {$k\leqslant 1$};
\draw (314,168.4) node [anchor=north west][inner sep=0.75pt, font=\footnotesize]  {$\Id$};
\draw (377,117.4) node [anchor=north west][inner sep=0.75pt, font=\footnotesize]  {$J_{A}$};
\draw (358.05,143.66) node [anchor=north west][inner sep=0.75pt, font=\footnotesize]  {$J_{\alpha Y_{1-\alpha }( A)}$};

\end{tikzpicture}
\end{center}

\begin{corollary}\label{Yosida3}
For $\mu>0$, the following hold:
\begin{enumerate}
\item\label{c:yosida:r}
$J_{Y_\mu(A)}=\frac{\mu}{\mu+1} \Id+\frac{1}{\mu+1} J_{(\mu+1) A}$.
\item \label{c:yosida:m}
$k(J_{Y_\mu(A)})=\frac{1}{\mu+1} k(J_{(\mu+1) A})$.
\end{enumerate}
\end{corollary}

\begin{example} Let $C$ be a nonempty closed convex set in $X$ and $C\neq X$. Consider the normal cone to $C$ defined by $N_{C}(x):=\{u \in X \mid \sup _{c \in C} \langle c-x, u\rangle \leqslant 0\}$ if $x \in C$, and $\varnothing$ otherwise.
Then
\begin{equation}\label{e:yu}
(\forall \mu>0)\ k(J_{Y_{\mu}(N_{C})})=k(J_{\mu^{-1}(\Id-P_{C})})=\frac{1}{2(\mu+1)}.
\end{equation}
In particular,
\begin{equation}\label{e:ya}
(\forall \alpha\in (0,1))\ k(J_{\alpha Y_{1-\alpha}(N_{C})})=k(J_{\alpha(1-\alpha)^{-1}(\Id-P_{C})})=\frac{\alpha}{2}.
\end{equation}
\end{example}
\begin{proof}
Apply Corollary~\ref{Yosida2} with $A=N_{C}$ to obtain
\begin{align}
J_{Y_{\mu}(N_{C})} &=J_{\mu^{-1}(\Id-J_{\mu N_{C}})}= J_{\mu^{-1}(\Id-P_{C})}\\
&=\frac{\mu}{\mu+1}\Id+\frac{1}{\mu+1}J_{(\mu+1)N_{C}}=\frac{\mu}{\mu+1}\Id+\frac{1}{\mu+1}J_{N_{C}}\\
&=\frac{\mu}{\mu+1}\Id+\frac{1}{\mu+1}P_{C}.
\end{align}
Using Lemma~\ref{multi} and $k(P_{C})=1/2$ because $C\neq X$, we have
$$k(J_{\mu^{-1}(\Id-P_{C})})=\frac{1}{\mu+1}k(P_{C})=\frac{1}{2(\mu+1)}.$$
Finally, \eqref{e:ya} follows from \eqref{e:yu} by using $\mu=(1-\alpha)/\alpha$.
\end{proof}

\begin{remark} Observe that Corollary~\ref{Yosida3}\ref{c:yosida:r} shows that
$Y_{\mu}(A)$ is the resolvent average of monotone operators
$0$ and $(\mu+1)A$; see, e.g., \cite{bartz1}.
\end{remark}

\subsection{Characterization of normally averaged mappings}\label{s:charc}

The Yosida regularization of monotone operators provides the key.
Recall that $T: X \rightarrow X$ is $\mu$-cocoercive with $\mu>0$
if $\mu T$ is firmly nonexpansive, i.e.,
$$
(\forall x \in X)(\forall y \in X) \quad\langle x-y,  T x-T y\rangle \geqslant \mu\|T x-T y\|^2.
$$

\begin{fact}\label{Yexist}\emph{\cite[Proposition 23.21(ii)]{BC}}
$T: X \rightarrow X$ is $\mu$-cocoercive if and only if there exists a maximally monotone operator $A: X\rightrightarrows X$ such that $T=Y_\mu(A)$.
\end{fact}

\begin{lemma}\label{ResThe} Let $A:X\rightrightarrows X$ be maximally monotone.
Suppose that $J_A$ is normally nonexpansive. Then $A$ is single-valued with full domain, and cocoercive.
\end{lemma}

\begin{proof} If $k(J_{A})=0$, Proposition~\ref{zero} shows that $J_{A}=\Id+ v$ for some $v\in X$. Then
$A:=-v$, which is clearly single-valued with full domain, and cocoercive.
Hence, we shall assume $0<k\left(J_A\right)<1/2$. Set $$N=\frac{J_A-\left(1-2 k\left(J_A\right)\right) \Id}{2 k\left(J_A\right)}.$$
Then $J_A=\left(1-2 k\left(J_A\right)\right) \Id+2 k\left(J_A\right) N$ and $N$ is nonexpansive with $k(N)=1/2$ by Lemma~\ref{multi}. It follows from Fact \ref{Minty}
that $N$ is firmly nonexpansive, i.e.,
there exists a maximally monotone operator $B:X\rightrightarrows X$
such that $N=J_B$. Thus by Corollary \ref{Yosida2} we have
$$
\begin{aligned}
J_A & =\left(1-2 k\left(J_A\right)\right) \Id+2 k\left(J_A\right) N
=\left(1-2 k\left(J_A\right)\right) \Id+2 k\left(J_A\right) J_B \\
& =J_{2 k\left(J_A\right) Y_{1-2 k\left(J_A\right)}(B)}.
\end{aligned}
$$
Therefore, $A=2 k\left(J_A\right) Y_{1-2 k\left(J_A\right)}(B)$. Since $J_A$ is normally nonexpansive, we have $2 k\left(J_A\right) \in (0,1)$. Thus $2 k\left(J_A\right) Y_{1-2 k\left(J_A\right)}(B)$, being a Yosida regularization,
is a single-valued, full domain, and cocoercive operator due to Fact \ref{Yexist}.
Hence $A$ is single-valued with full domain, and cocoercive.
\end{proof}

\begin{lemma}\label{ResInv}
Suppose $A: X\rightrightarrows X$ is single-valued with full domain, and cocoercive.
Then $J_A$ is normally nonexpansive.
\end{lemma}

\begin{proof}
Since $A$ is single-valued with full domain, and cocoercive, by Fact \ref{Yexist}, there exist
a maximally monotone operator $B:X\rightrightarrows X$ and $\mu>0$
such that $A=Y_\mu(B)$. Since $B$ is maximally monotone, by Corollary \ref{Yosida3}, we have
$$
J_{Y_\mu(B)}=\frac{\mu}{\mu+1} \Id+\frac{1}{\mu+1} J_{(\mu+1) B}=
J_A.
$$
Since $B$ is maximally monotone and $\mu+1>1$, we have $(\mu+1) B$ is maximally monotone as well. Thus $k(J_{(\mu+1) B}) \leqslant 1/2$ by Fact \ref{Minty}. Now Lemma \ref{multi} gives
$$
k\left(J_A\right)=\frac{1}{\mu+1} k\left(J_{(\mu+1) B}\right) \leqslant \frac{1}{\mu+1} \cdot \frac{1}{2}<\frac{1}{2}.
$$
\end{proof}

The main result of this subsection comes as follows.

\begin{theorem}[characterization of normally averaged mapping]
\label{Charac1}
Let $A:X\rightrightarrows X$ be maximally monotone.
Then $J_A$ is normally nonexpansive if and only if $A$ is single-valued with full domain, and cocoercive.
\end{theorem}

\begin{proof}
Combine Lemma \ref{ResThe} and Lemma \ref{ResInv}.
\end{proof}

In view of Fact \ref{Minty}, the characterization of special nonexpansiveness follows immediately as well.

\begin{example}
Let $A\in \SSS_{++}^{n}$, the set of $n\times n$  positive
definite symmetric matrices. Then $k\left(J_A\right)<1/2$ and $k\left(J_{A^{-1}}\right)<1/2$ by Theorem~\ref{Charac1}.
\end{example}

The following fact follows from \cite[Theorem 2.1(i)\&(iv)]{moffat}.
\begin{fact}\label{f:fulldomain}
Let $T:X\rightarrow X$ be firmly nonexpansive.
Then the following hold:
\begin{enumerate}
\item $T=J_{A}$ for a maximally monotone operator $A:X\rightrightarrows X$.
\item $T$ is injective if and only if $A$ is at most single-valued, i.e., $$(\forall x\in\dom A)\ Ax
\text{ is empty or a singleton.}$$
\item $T$ is surjective if and only if $\dom A=X$.
\end{enumerate}
\end{fact}

\begin{remark}
Combining Theorem \ref{Charac1}, Fact \ref{f:fulldomain} and Fact \ref{Minty}, we recover Theorem \ref{bijectivethe}, since $A$ being cocoercive implies $\Id+A$ being Lipschitz.
\end{remark}

\section{Proximal operator}\label{s:proximal}

Let $f \in \Gamma_0(X).$ Recall that the proximal operator of $f$ is given by
$$\mathrm{P}_{f}(x): = \underset{u\in X}{\operatorname{argmin}}\left\{f(u)+\frac{1}{2}\|u-x\|^2\right\},$$
that the Moreau envelope of $f$ with parameter $\mu>0$ is defined by $e_\mu f(x):=\min _{u \in X}(f(u)+\frac{1}{2 \mu}\|u-x\|^2)$, and that the Fenchel conjugate of $f$ is defined by $f^*(y):=\sup _{x \in X}(\langle x, y\rangle-f(x))$ for $y \in X$. It is well known that $\mathrm{P}_{f}=(\mathrm{Id}+\partial f)^{-1}$,
where $\partial f$
is the subdifferential of $f$ given by
$\partial f(x):=\{u \in X \mid(\forall y \in X) \
f(y)\geqslant f(x)+\langle u, y-x\rangle\}$ if $x\in \dom f$, and
$\varnothing$ if $x\not\in\dom f$. Also, $\mathrm{P}_f$ is firmly nonexpansive, i.e., $k\left(\mathrm{P}_f\right) \leqslant 1/2$. See, e.g., \cite{BC,RW}.

In this section, we will characterize the normal and special nonexpansiveness of $\mathrm{P}_{f}$.
We begin with
\begin{definition}[L-smoothness]
Let $L \in [0,+\infty)$. Then $f$ is $L$-smooth on $X$ if $f$ is
Fr\'echet differentiable on $X$ and $\nabla f$ is $L$-Lipschitz, i.e.,
$$
(\forall x \in X)(\forall y \in X) \quad\|\nabla f(x)-\nabla f(y)\| \leqslant L\|x-y\|.
$$
\end{definition}

\begin{fact}[Baillon-Haddad]\label{coL} \emph{\cite{BH} (see also \cite[Corollary 18.17]{BC})}
Let $f \in \Gamma_0(X)$. Suppose $f$ is Fr\'echet differentiable on $X$. Then $\nabla f$ is $\mu$-cocoercive if and only if $\nabla f$ is $\mu^{-1}$-Lipschitz continuous.
\end{fact}
For further properties of $L$-smooth functions, see \cite{BC,AB, RW}. We also need
\begin{fact}[Moreau]\emph{\cite[Theorem 20.25]{BC}} \label{Mor}
Let $f \in \Gamma_0(X)$. Then $\partial f$ is maximally monotone.
\end{fact}

The following interesting result characterizes a $L$-smooth function $f$ via the modulus of averagedness of $\mathrm{P}_{f}$.
It shows that for proximal operators not only can Theorem~\ref{Charac1} be significantly improved but also the converse of Theorem~\ref{bijectivethe} holds.

\begin{theorem}[characterization of normal proximal operator]
\label{t:lipsmooth}
Let $f\in\Gamma_{0}(X)$. Then the following are equivalent:
\begin{enumerate}
\item \label{i:normal}
$\mathrm{P}_{f}$ is normally nonexpansive.
\item\label{i:lsmooth} There exists $L>0$ such that $f$ is $L$-smooth on $X$.
\item\label{i:sconvex} $f^*$ is $1/L$-strongly convex for some $L>0$.
\item\label{i:contraction} $\mathrm{P}_{f^*}$ is a Banach contraction.
\item\label{i:bi-lip} $\mathrm{P}_{f}$ is a bi-Lipschitz homeomorphism of $X$.
\end{enumerate}
\end{theorem}

\begin{proof} ``\ref{i:normal}$\Leftrightarrow$\ref{i:lsmooth}":
By Fact \ref{Mor}, $\partial f$ is maximally monotone. Let $A=\partial f$ in Theorem \ref{Charac1} and combine it with Fact \ref{coL}.

``\ref{i:lsmooth}$\Leftrightarrow$\ref{i:sconvex}": Apply \cite[Theorem 18.15]{BC}.

``\ref{i:sconvex}$\Leftrightarrow$\ref{i:contraction}": Apply \cite[Corollary 3.6]{luo-wy}.

``\ref{i:normal}$\Rightarrow$\ref{i:bi-lip}": Apply Theorem~\ref{bijectivethe}.

``\ref{i:bi-lip}$\Rightarrow$\ref{i:normal}": The assumption implies that $(\mathrm{P}_{f})^{-1}=\Id+\partial f$
is full domain, single-valued and Lipschitz, so is
$\partial f=\nabla f$. By Fact~\ref{coL}, $\nabla f$ is co-coercive. It remains to
apply Theorem~\ref{Charac1}.
\end{proof}

\begin{remark}
(1) Bi-Lipschitz homeomorphisms of a Euclidean space form an important class of operators. For instance, Hausdorff dimension, which plays a central role in fractal geometry and harmonic analysis, is bi-Lipschitz invariant (see \cite{falconer-book}). Theorem \ref{t:lipsmooth}$(\mathrm{i})\Leftrightarrow(\mathrm{v})$ thus provides a large class of such nonlinear operators.

(2)
By endowing $\Gamma_{0}(X)$ with the topology of epi-convergence, see, e.g.,
\cite[Proposition 3.5, Corollary 4.18]{planiden},
Theorem~\ref{t:lipsmooth} $(\mathrm{i}) \Leftrightarrow(\mathrm{ii})$ implies
that \emph{most} convex functions have their proximal mappings with modulus of averagedness exactly
$1/2$, in the sense of co-meagerness (the complement of a meager set).
\end{remark}

The characterization of special proximal operator follows immediately as well. The following example shows that
$\mathrm{P}_{f}$ being only bijective
does not imply that $\mathrm{P}_{f}$ is normally nonexpansive.

\begin{example}\label{e:bijectiveno}
Let $X=\mathbb{R}$. Define $$\varphi(x):=\begin{cases}
\ln x & \text{ if $x \geqslant e$,} \\
\frac{1}{e} x & \text{ if $-e<x<e$,}\\
-\ln (-x) & \text{ if $x \leqslant-e.$}
 \end{cases}
 $$
Then the following hold:
\begin{enumerate}
\item\label{i:prox:op} $\varphi$ is a proximal operator of a function in $\Gamma_0(\mathbb{R})$.
\item\label{i:bij} $\varphi$ is a bijection.
\item\label{i:snon} $\varphi$ is specially nonexpansive.
\item\label{i:notlip} The inverse mapping of $\varphi$:
$$(\varphi)^{-1}(y)=\begin{cases}
e^y &\text{ if $y\geqslant 1$,}\\
e y &\text{ if $-1\leqslant y\leqslant 1$,}\\
-e^{-y} & \text{ if $y\leqslant -1$,}
\end{cases}
$$
is not Lipschitz.
\end{enumerate}
 \end{example}
\begin{proof}
\ref{i:prox:op}: $\varphi$ is a proximal operator because it is  nonexpansive and increasing (see \cite[Proposition 24.31]{BC}).
\ref{i:bij}: Obvious. \ref{i:snon}: We have that $\varphi$ is differentiable with
$$\varphi^{\prime}(x)= \begin{cases}
\frac{1}{x} & \text{ if $x \geqslant e$,}\\
\frac{1}{e} & \text{ if $-e<x<e$,}\\
-\frac{1}{x} & \text{ if $x \leqslant-e.$}
\end{cases}$$
Thus $\inf _{x \in \mathbb{R}} \varphi^{\prime}(x)=0$. By Lemma~\ref{modulusR} or \cite[Proposition 2.8]{BBM}, $k(\varphi)=(1-\inf _{x \in \mathbb{R}} \varphi^{\prime}(x))/{2}=1/2$.
\ref{i:notlip}: Direct calculations.
\end{proof}

\begin{corollary}\label{Pfspe}
Let $f\in\Gamma_{0}(X)$. Suppose $\dom f \neq X$. Then $\mathrm{P}_{f}$ is specially nonexpansive.
\end{corollary}

\begin{proof}
Observe that $\operatorname{dom} f \neq X$ implies $\operatorname{dom} \partial f \neq X$. Thus $f$ is not $L$-smooth for any $L>0$ and the result follows by Theorem~\ref{t:lipsmooth}.
\end{proof}

\begin{remark}
When $C$ is a nonempty closed convex subset of $X$ and $C\neq X$, obviously $\iota_C \in \Gamma_0(X)$ and $\operatorname{dom} \iota_C=C$. By Corollary \ref{Pfspe}, $P_C$ is specially nonexpansive,
which recovers Example \ref{projectionaverage}.
\end{remark}

For the Moreau envelope we have the following result.

\begin{proposition}\label{t:envel}
Let $f\in\Gamma_{0}(X)$ and let $\mu, \alpha>0$. Then
\begin{equation}\label{e:mod:env}
k(\mathrm{P}_{\alpha e_{\mu}f})=\frac{\alpha}{\mu+\alpha}k(\mathrm{P}_{(\mu+\alpha)f}).
\end{equation}
If, in addition, $f$ is not Lipschitz smooth, then
$$k(\mathrm{P}_{\alpha e_{\mu}f})=\frac{1}{2}\frac{\alpha}{\mu+\alpha}.$$
\end{proposition}
\begin{proof}
By \cite[Theorem 27.9]{BM}, we have
$$\mathrm{P}_{\alpha e_{\mu}f}=\frac{\mu}{\mu+\alpha}\Id+\frac{\alpha}{\mu+\alpha}\mathrm{P}_{(\mu+\alpha)f}.$$
It suffices to apply Lemma~\ref{multi}.

If, in addition, $f$ is not Lipschitz smooth, then $(\mu+\alpha)f$ is not Lipschitz smooth
so that $k(\mathrm{P}_{(\mu+\alpha)f})=1/2$ by Theorem~\ref{t:lipsmooth}. Use \eqref{e:mod:env}
to complete the proof.
\end{proof}

\begin{example} Let $\mu, \alpha>0$.
Consider the Huber function defined by
$$
H_{\mu}:X\rightarrow\RR: x\mapsto \begin{cases}
\frac{1}{2\mu}\|x\|^2 & \text{ if $\|x\|\leqslant \mu$,}\\
\|x\|-\frac{\mu}{2} & \text{ if $\|x\|>\mu$.}
\end{cases}
$$
It is well-known that $H_{\mu}=e_{\mu}\|\cdot\|$ and that $\|\cdot\|$ is not Lipschitz smooth,
Therefore, by Proposition~\ref{t:envel},
$$k(\mathrm{P}_{\alpha H_{\mu}})=\frac{1}{2}\frac{\alpha}{\mu+\alpha}.$$
\end{example}

\begin{example}
Let $C$ be a nonempty closed convex subset of $X$ and $C\neq X$. Consider the support function of $C$ defined by
$\sigma_C: X \rightarrow[-\infty,+\infty]: x \mapsto \sup _{c \in C}\langle c, x\rangle.$
Then the following hold:
\begin{enumerate}
\item\label{i:single}If $C$ is a singleton, then
$k(P_{C})=1/2$ and
$(\forall \lambda>0)\ k(\mathrm{P}_{\lambda\sigma_{C}})=0.$
\item\label{i:notsingle} If $C$ contains more than one point, then
$k(P_{C})=1/2$ and
$(\forall \lambda>0)\ k(\mathrm{P}_{\lambda\sigma_{C}})=\frac{1}{2}.$
\end{enumerate}
\end{example}
\begin{proof} The fact that $k(P_{C})=1/2$ has been given by Example~\ref{projectionaverage}.
Now observe that the support function $\sigma_{C}$ has
$\mathrm{P}_{\lambda\sigma_{C}}=\Id-\lambda P_{C}(\cdot/\lambda).$

\ref{i:single}: We have $\mathrm{P}_{\lambda\sigma_{C}}=\Id+v$ for some $v \in X$. Then apply Proposition~\ref{zero}.

\ref{i:notsingle}: The function $\lambda \sigma_{C}(x)$ is not Lipschitz smooth, since it is not
differentiable at $0$. Apply Theorem~\ref{t:lipsmooth} to derive
$k(\mathrm{P}_{\lambda \sigma_{C}})=1/2$.
\end{proof}

\section{Compute modulus of averagedness via other constants or values}\label{s:compute}
In this section, introducing monotone value for monotone operators, cocercive
value for cocoercive mappings and Lipschitz value for Lipschitz mappings,
we provide various formulae to quantify the modulus of averagedness for resolvents and
proximal operators.

\subsection{Monotone value and cocoercive value}

Recall that we assume $A: X \rightrightarrows X$ is a maximally monotone operator.
For $\mu>0$, we say that $A$ is $\mu$-strongly monotone if $A-\mu \mathrm{Id}$ is monotone, i.e., $$
(\forall(x, u) \in \operatorname{gra} A)(\forall(y, v) \in \operatorname{gra} A) \quad\langle x-y, u-v\rangle \geqslant \mu\|x-y\|^2 .
$$
It is clear that if an operator is $\mu_0$-strongly monotone (or cocoercive), then it is $\mu$-strongly monotone (or cocoercive) for $\mu \leqslant \mu_0$. Observing this property, we define the following functions for a maximally monotone operator.

\begin{definition}[monotone value]
Suppose that A is strongly monotone. The monotone value (or best strong monotonicity constant) of $A$ is defined by
$$
m(A):=\sup \{\mu>0 \mid A \text { is } \mu \text {-strongly monotone }\}.
$$
Otherwise, we define $m(A)=0$.
\end{definition}

\begin{definition}[cocoercive value]
Suppose A is single-valued with full domain, and cocoercive. The cocoercive value (or best
cocoercivity constant) of $A$ is defined by
$$
c(A):=\sup \{\mu>0 \mid A \text { is } \mu \text {-cocoercive }\}.
$$
Otherwise, we define $c(A)=0$.
\end{definition}

We present basic properties of monotone value and cocoercive value. Note an operator is $\mu$-cocoercive if and only if its inverse is $\mu$-strongly monotone.

\begin{proposition}\label{cAp}
Let $\mu > 0$. The following hold:
\begin{enumerate}
\item\label{i:co:mono}\emph{(\textbf{duality})}
$m\left(A\right)=c\left(A^{-1}\right)$ and $m\left(A^{-1}\right)=c\left(A\right)$.
\item\label{i:mono:val1}$m(\mu A)=\mu m(A)$ and $c(\mu A)=\mu^{-1} c(A)$.
\item\label{cA=inft}$c(A)=+\infty$ if and only if $A$ is a constant operator on $X$.
\item\label{i:mono:val3}$m(A+B) \geqslant m(A)+m(B)$ and $$c\left(\left(A^{-1}+B^{-1}\right)^{-1}\right) \geqslant c(A)+c(B).$$
\item\label{i:mono:val2}\emph{(\textbf{Yosida regularization})} $m(A+\mu\Id)=m(A)+\mu$ and $c\left(Y_\mu(A)\right)=c(A)+\mu.$
\end{enumerate}
\end{proposition}

\begin{proof}
\ref{i:co:mono}, \ref{i:mono:val1}, \ref{cA=inft} and \ref{i:mono:val3} can be directly verified. \ref{i:mono:val2}:
Since $Y_\mu(A)=\left(\mu \mathrm{Id}+A^{-1}\right)^{-1}$, we have
$$
\begin{aligned}
c\left(Y_\mu(A)\right)
& =c\left((\mu \mathrm{Id}+A^{-1}\right)^{-1})
=m\left(\mu \mathrm{Id}+A^{-1}\right) \\
& =\mu+m\left(A^{-1}\right)
 =\mu+c(A).
\end{aligned}
$$
\end{proof}

The following fact connects averaged operators with cocoercive mappings, and can be directly verified.
\begin{fact} \label{f:xu} \emph{\cite[Proposition 3.4(iii)]{xu}}
Let $T:X\rightarrow X$ be nonexpansive and $\alpha \in (0, 1)$. Then $T$ is $\alpha$-averaged if and only if $\Id-T$ is $1/(2\alpha)$-cocoercive.
\end{fact}

\begin{proposition}\label{p:coercive:const}
Let $T:X\rightarrow X$ be nonexpansive. Then
$$k(T)=\frac{1}{2c(\Id-T)}.$$
\end{proposition}
\begin{proof}
Combine Proposition~\ref{zero} and Fact~\ref{f:xu}.
\end{proof}

\begin{corollary}\label{ellT<=2kT}
Let $T: X \rightarrow X$ be normally nonexpansive. Then $\Id-T$ is a Banach  contraction with constant $2 k(T)$.
\end{corollary}

\begin{proof}
By Proposition \ref{p:coercive:const}, $\operatorname{Id}-T$ is cocoercive with constant $1/(2 k(T))$. Using the Cauchy-Schwarz inequality, we have that $\Id-T$ is Lipschitz with constant $2 k(T)$. The contraction property follows by $2 k(T)<1$ since $T$ is normally nonexpansive.
\end{proof}

\begin{remark}\label{suggestion}
Lemma~\ref{main1} can also be proved by using Corollary \ref{ellT<=2kT} and the Banach fixed-point theorem. Indeed, given a normally nonexpansive $T$ and for any $v \in X$, the mapping $x \mapsto x-T x+v$ is
a Banach contraction and therefore has a fixed point $x_0$. Then $x_0=x_0-T x_0+v$ which implies that $T x_0=v$, therefore $T$ is surjective.
\end{remark}

The following result connects the modulus of averagedness
of a resolvent to the co-coercivity of associated maximally monotone operator.

\begin{proposition}[modulus of averagedness via cocoercive value]\label{Mainformula1}
Let $A: X \rightrightarrows X$ be maximally monotone and $\alpha>0$. Then
$$
k\left(J_{\alpha A}\right)=\frac{1}{2} \frac{\alpha}{\alpha+c(A)}.
$$
\end{proposition}

\begin{proof}
In view of Proposition~\ref{cAp}\ref{i:mono:val1}, it suffices to prove the case when $\alpha=1$. Note that $Y_1(A)=\operatorname{Id}-J_{A}$. By Proposition~\ref{cAp}\ref{i:mono:val2}, $c\left(\operatorname{Id}-J_{A}\right)=c(A)+1$. Now apply Proposition \ref{p:coercive:const}.
\end{proof}

We have the following corollary in view of Proposition~\ref{cAp}\ref{i:co:mono}.

\begin{corollary}[modulus of averagedness via monotone value]
Let $A:X\rightrightarrows X$ be maximally monotone and $\alpha>0$. Then
$$k(J_{\alpha A})=\frac{1}{2} \frac{\alpha}{\alpha+m\left(A^{-1}\right)}.$$
\end{corollary}

The following example illustrates our formulae in this subsection.

\begin{example} Suppose that $A:X\rightarrow X$ is a bounded linear operator and that
$A$ is skew, i.e., $(\forall x\in X) \scal{x}{Ax}=0.$ Then $A$ is maximally monotone,
and the following hold:
\begin{enumerate}
\item If $A\equiv 0$, then $c(A)=+\infty$. Clearly $$k(J_{A})=k(\Id)=0=\frac{1}{2}\frac{1}{1+\infty}.$$
\item If $A$ is not a zero operator, then $c(A)=m(A)=m(A^{-1})=c(A^{-1})=0$. Therefore the formulae give $k(J_{A})=k(J_{A^{-1}})=1/2$, which coincides with Theorem~\ref{Charac1} because $A$ and $A^{-1}$ is not cocoercive.
\end{enumerate}
\end{example}

\subsection{Lipschitz value}

\begin{definition}[Lipschitz value]
Let $T: X \rightarrow X$. The Lipschitz value (or best Lipschitz constant) of $T$ is defined by
$$
\ell(T):=\inf\left\{L \geqslant 0 \mid \forall x, y \in X,  \|Tx-Ty\| \leqslant L\|x-y\|\right\}.
$$
Moreover, for a maximally monotone operator $A: X \rightrightarrows X$, define $\ell(A)=+\infty$ if $A$ is not single-valued with full domain.
\end{definition}

The following formula connects Lipschitz value with cocoercive value. Note that we follow the convention that $\inf \varnothing=+\infty$, $(+\infty)^{-1}=0$ and $0^{-1}=+\infty$.

\begin{lemma}
$\ell(A) \leqslant[c(A)]^{-1}$.
\end{lemma}

\begin{proof}
Suppose $c(A) \in(0,+\infty)$. Then $A$ is $c(A)$-cocoercive and therefore $[c(A)]^{-1}$-Lipschitz on $X$
by the Cauchy-Schwarz inequality. Thus $\ell(A) \leqslant[c(A)]^{-1}$.

Suppose $c(A)=+\infty$. It follows from Proposition~\ref{cAp}\ref{cA=inft} that
$A$ is a constant operator. Thus $\ell(A)=0=[c(A)]^{-1}$.

Suppose $c(A)=0$. Then $\ell(A) \leqslant+\infty=[c(A)]^{-1}$.
\end{proof}

\begin{fact}\label{dif}\emph{\cite[Proposition 17.31]{BC}}
Let $f$ be convex and proper on $X$, and suppose that $x \in \operatorname{int} \operatorname{dom} f$. Then $
f \text { is Gâteaux differentiable at } x \Leftrightarrow \partial f(x) \text { is a singleton }
$
in which case $\partial f(x)=\{\nabla f(x)\}$.
\end{fact}

\begin{proposition}\label{clformula2}
$\ell(\partial f)=[c(\partial f)]^{-1}$.
\end{proposition}

\begin{proof}
Suppose $c(\partial f) \in(0,+\infty)$. Then $\partial f$ is singe-valued with full domain. Thus $\partial f=\nabla f$ by Fact \ref{dif}. While $\nabla f$ is $c(\partial f)$-cocoercive, by applying Fact \ref{coL} we have $\ell(\nabla f)=[c(\nabla f)]^{-1}$.

Suppose $c(\partial f)=+\infty$. Then $\partial f$ is a constant operator by Proposition \ref{cAp}. Thus $\ell(\partial f)=0=[c(\partial f)]^{-1}$.

Suppose $c(\partial f)=0$. If $\partial f$ is singe-valued with full domain, then again by applying Fact \ref{dif} and Fact \ref{coL} we have $\partial f=\nabla f$ is not Lipschitz, thus $\ell(\partial f)=+\infty=[c(\partial f)]^{-1}$. If $\partial f$ is not singe-valued, or not with full domain, then $\ell(\partial f)=+\infty$ by the definition of Lipschitz value. Thus $\ell(\partial f)=+\infty=[c(\partial f)]^{-1}$.
\end{proof}

Now we are able to propose the following interesting formula for proximal operators.

\begin{theorem}[modulus of averagedness via Lipschitz value]
\label{mainformula2}
Let $f \in \Gamma_0(X)$. Then
$$
k\left(\mathrm{P}_{f}\right)=\frac{1}{2} \frac{1}{1+[\ell(\partial f)]^{-1}}.
$$
\end{theorem}

\begin{proof}
By Fact \ref{Mor}, $\partial f$ is maximally monotone. The result follows by letting $A=\partial f$ in Proposition \ref{Mainformula1} and combining it with Proposition \ref{clformula2}.
\end{proof}

Using $\ell(\alpha T)=\alpha \ell(T)$ for $\alpha>0$, we obtain the following result.

\begin{corollary}\label{Mainformula3}
Let $f \in \Gamma_0(X)$ be $L$-smooth on $X$ for some $L>0$ and let $\alpha>0$. Then
$$
k\left(\mathrm{P}_{\alpha f}\right)=\frac{1}{2} \frac{\alpha \ell(\nabla f)}{1+\alpha \ell(\nabla f)}.
$$
\end{corollary}

The following example illustrates our formulae in this subsection.

\begin{example}\label{example1}
Let $C$ be a nonempty closed convex set in $X$ and $C \neq X$. Consider the distance function of $C$ defined by
$
d_C(x): X \rightarrow[-\infty,+\infty]: x \mapsto\inf _{c \in C}\|x-c\|.
$
Then for any $\alpha>0$ the following hold:
\begin{enumerate}
\item\label{i:dsquare} $k\left(\mathrm{P}_{\frac{\alpha}{2} d_C^2}\right)=\frac{1}{2} \frac{\alpha}{1+\alpha}$.
\item\label{i:constantco} $c\left(\Id-P_C\right)=\ell\left(\Id-P_C\right)=1$.
\item\label{i:proxbiject} $\mathrm{P}_{\frac{\alpha}{2}d_C^2}$ is a bi-Lipschitz homeomorphism of $X$.
\end{enumerate}
\end{example}

\begin{proof}
\ref{i:dsquare}: By \cite[Example 6.65]{AB}, $\mathrm{P}_{\frac{\alpha}{2} d_C^2}=\frac{1}{1+\alpha} \Id+\frac{\alpha}{1+\alpha} P_C$. Thus we have $k\left(\mathrm{P}_{\frac{\alpha}{2} d_C^2}\right)=\frac{\alpha}{1+\alpha}k(P_C)=\frac{1}{2} \frac{\alpha}{1+\alpha}$ by Lemma \ref{multi} and Example \ref{projectionaverage}.

\ref{i:constantco}:
We have  $c\left(\mathrm{Id}-P_C\right)=1$ by Proposition \ref{p:coercive:const} and $k(P_C)=1/2$. On the other hand, since $\frac{1}{2} d_C^2 \in \Gamma_0(X)$ and $\nabla \frac{1}{2} d_C^2=\Id-P_C$,
see, e.g., \cite[Corollary 12.31]{BC}, we have  $c\left(\mathrm{Id}-P_C\right)=\ell\left(\mathrm{Id}-P_C\right)=1$ by Proposition \ref{clformula2}.

Consequently, Corollary \ref{Mainformula3} is verified by the results of \ref{i:dsquare} and \ref{i:constantco}:
$$
k\left(\mathrm{P}_{\frac{\alpha}{2} d_C^2}\right)=\frac{1}{2} \frac{\alpha \ell(\nabla \frac{1}{2} d_C^2)}{1+\alpha \ell(\nabla \frac{1}{2} d_C^2)}=\frac{1}{2} \frac{\alpha \ell(\mathrm{Id}-P_C)}{1+\alpha \ell(\mathrm{Id}-P_C)}=\frac{1}{2} \frac{\alpha}{1+\alpha}.
$$

\ref{i:proxbiject}: By \ref{i:dsquare}, $k\left(\mathrm{P}_{\frac{\alpha}{2} d_C^2}\right)=\frac{1}{2} \frac{\alpha}{1+\alpha}<\frac{1}{2}$, i.e., $\mathrm{P}_{\frac{\alpha}{2} d_C^2}$ is normally nonexpansive. The result follows by Theorem \ref{bijectivethe}.
\end{proof}

\section{Bauschke, Bendit \& Moursi's example generalized}\label{s:twod}
The following example on the modulus of averagedness of $P_{V}P_{U}$ extends \cite[Example 3.5]{BBM} in $\RR^2$
to a Hilbert space. Instead of using \cite[Theorem 3.2]{BBM}, we provide a much simpler proof.
\begin{example}\label{e:fried}
Let $\theta\in (0,\pi/2)$. In the product Hilbert space
$\HH=X\times X$, define
$$U=X\times \{0\},\quad V=\menge{(y,(\tan\theta) y)}{y\in X}.$$
Then
\begin{equation}\label{e:angle}
k(P_VP_{U})=\frac{1+\cos\theta}{2+\cos\theta}.
\end{equation}
\end{example}
\begin{proof}
We have
$$P_{U}=\begin{bmatrix}
\Id & 0\\
0 & 0
\end{bmatrix},
\text{ and }
P_{V}=\begin{bmatrix}
\frac{1}{1+\tan^2\theta}\Id & \frac{\tan\theta}{1+\tan^2\theta}\Id\\
\frac{\tan\theta}{1+\tan^2\theta}\Id & \frac{\tan^2\theta}{1+\tan^2\theta}\Id
\end{bmatrix}
$$
so that
$$P_VP_U=\begin{bmatrix}
\frac{1}{1+\tan^2\theta}\Id & 0\\
\frac{\tan\theta}{1+\tan^2\theta}\Id & 0
\end{bmatrix}.
$$
Put $T=P_VP_U$. Then $T$ is $k$-averaged  if and only if
\begin{equation}\label{e:linear:ave}
(\forall x\in\HH)\ \|Tx\|^2+(1-2k)\|x\|^2\leqslant 2(1-k)\scal{x}{Tx}.
\end{equation}
For $x=(x_{1}, x_{2})$ with $x_i\in X$,
we have
$$Tx=\bigg(\frac{1}{1+\tan^2\theta}x_1, \frac{\tan\theta}{1+\tan^2\theta} x_1\bigg),
\quad \scal{Tx}{x}=\frac{\|x_1\|^2}{1+\tan^2\theta}+\frac{\tan\theta}{1+\tan^2\theta}\scal{x_{1}}{x_{2}}.$$
Substitute above into \eqref{e:linear:ave} to obtain
$$\frac{\|x_{1}\|^2}{1+\tan^2\theta}+(1-2k)(\|x_{1}\|^2+\|x_{2}\|^2) \leqslant 2(1-k)\bigg(\frac{\|x_1\|^2}{1+\tan^2\theta}+\frac{\tan\theta}{1+\tan^2\theta}\scal{x_{1}}{x_{2}}\bigg),$$
which can be simplified to
\begin{align}\label{e:bigger}
(2k-1)\frac{-\tan^2\theta}{1+\tan^2\theta}\|x_{1}\|^2+(1-2k)\|x_{2}\|^2-2(1-k)\frac{\tan\theta}{1+\tan^2\theta}
\scal{x_{1}}{x_{2}} & \leqslant 0.
\end{align}
When $x_{2}=0$, this gives
$(2k-1)(-\tan^2\theta)\leqslant 0$, so $k\geqslant 1/2$. If $k=1/2$, this gives
$$(\forall x_{1},x_{2}\in X)\ -\frac{\tan\theta}{1+\tan^2\theta}\scal{x_{1}}{x_{2}}\leqslant 0,$$
which is impossible. Thus $k>1/2$. Dividing \eqref{e:bigger} by $\|x_{2}\|^2$ and applying
the Cauchy-Schwarz inequality, we have
\begin{equation}\label{e:quadratic}
(2k-1)\frac{\tan^2\theta}{1+\tan^2\theta}\bigg(\frac{\|x_{1}\|}{\|x_{2}\|}\bigg)^2
+2(1-k)\frac{\tan\theta}{1+\tan^2\theta}\bigg(\pm\frac{\|x_{1}\|}{\|x_{2}\|}\bigg)+(2k-1) \geqslant 0.
\end{equation}
Substituting $t=\|x_{1}\|/\|x_{2}\|$ into \eqref{e:quadratic} yields
$$(2k-1)\frac{\tan^2\theta}{1+\tan^2\theta}t^2
+2(1-k)\frac{\tan\theta}{1+\tan^2\theta}(\pm t)+(2k-1) \geqslant 0$$
which happens if and only if
$$\bigg(2(1-k)\frac{\tan\theta}{1+\tan^2\theta}\bigg)^2 \leqslant 4(2k-1)^2\frac{\tan^2\theta}{1+\tan^2\theta}$$
i.e., $(1-k)^2\leqslant (2k-1)^2(1+\tan^2\theta)$. Taking square root both sides, we have
$1-k\leqslant (2k-1)/\cos\theta$, so that $k\geqslant (1+\cos\theta)/(2+\cos\theta)$. Hence
$$k(T)=\frac{1+\cos\theta}{2+\cos\theta}.$$
\end{proof}
\begin{remark} Let
$U, V$ be two closed subspaces of
$\HH$. Recall that while the cosine of Dixmier angle
between $U, V$
is defined by
\begin{equation}
c_{D}(U,V)=\sup \menge{\scal{u}{v}}{u\in U, v\in V, \|u\|\leqslant 1, \|v\|\leqslant 1},
\end{equation}
the cosine of the Friedrich angle between $U, V$ is defined by
\begin{equation}
c_{F}(U,V)=\sup \menge{\scal{u}{v}}{u\in U\cap (U\cap V)^{\perp}, v\in V\cap (U\cap V)^{\perp}, \|u\|\leqslant 1, \|v\|\leqslant 1}.
\end{equation}
For more details on the angle between subspaces, see \cite{BBM, deutsch}.
With $U=X\times \{0\},\quad V=\menge{(y,(\tan\theta) y)}{y\in X}$ given in Example~\ref{e:fried},
for $\theta\in (0,\pi/2)$, we have $U\cap V=0$ so that $(U\cap V)^{\perp}=\HH=X\times X$. Then
\begin{align}
c_{D}(U,V)=c_{F}(U,V)
&=\menge{\scal{(x,0)}{(y,(\tan\theta) y)}}{x\in X, y\in X, \|x\|\leqslant 1, \|(y,(\tan\theta) y)\|\leqslant 1}\\
&= \menge{\scal{x}{y}}{x\in X, y\in X, \|x\|\leqslant 1, \|y\|\leqslant \cos\theta}
=\cos\theta.
\end{align}
Hence both the Dixmier and Friedrich angles between $U$ and $V$ are exactly $\theta$.
\end{remark}

\section*{Acknowledgments}
The authors thank the editor and an anonymous referee for careful reading and
constructive comments, especially on Example~\ref{lim:counterexample} and Remark~\ref{suggestion}. 
Inspiring discussions with Dr.\ H.H.\ Bauschke benefited the paper.
The research of the authors was partially supported by the Natural Sciences and Engineering Research Council of Canada.

\end{document}